\documentclass[12pt]{article}

\usepackage{amsmath,amsthm,amssymb}
\usepackage{extpfeil,extarrows}
\usepackage[colorlinks,
            linkcolor=red,
            anchorcolor=blue,
            citecolor=magenta
            ]{hyperref}
\usepackage{pdfsync}
\usepackage[numbers,sort&compress]{natbib}
\allowdisplaybreaks

\theoremstyle{definition}

\theoremstyle{plain}
\newtheorem{theorem}{Theorem}[section]
\newtheorem{definition}{Definition}[section]
\newtheorem{remark}{Remark}[section]
\newtheorem{lemma}{Lemma}[section]

\newtheorem{proposition}{Proposition}[section]

\numberwithin{equation}{section}

\renewcommand{\ni}{\noindent}

\newcommand{\vs}{\vspace}

\begin{document}

\title{Existence and dependency results for coupled Schr\"odinger equations
with critical exponent on waveguide manifold\footnote{  This work was partially supported by NNSFC (No. 12171493).}}

\author{ Jun Wang$^{a}$\footnote {Corresponding author. wangj937@mail2.sysu.edu.cn (J. Wang), mcsyzy@mail.sysu.edu.cn (Z. Yin)}, Zhaoyang Yin$^{a, b}$ \\
{\small $^{a}$Department of Mathematics, Sun Yat-sen University, Guangzhou, 510275, China } \\
{\small $^{b}$School of Science, Shenzhen Campus of Sun Yat-sen University, Shenzhen, 518107, China } \\
}

	\date{}

	\maketitle

\date{}

 \maketitle \vs{-.7cm}

  \begin{abstract}
We study the coupled Schr\"odinger equations with critical exponent on $\mathbb{R}^3  \times \mathbb{T}$. With the help of scaling argument and semivirial-vanishing technology, we obtain the existence and $y$-dependence  of solution, the tori can be generalized to $1$-dimensional compact Riemannian manifold.
Moreover, the conclusion of this paper can be extended to systems with any number of components.
\end{abstract}

{\footnotesize {\bf   Keywords:}  Coupled Schr\"odinger equations; Normalized solutions; Waveguide manifold.

{\bf 2010 MSC:}  35A15, 35B38, 35R01.
}

\section{ Introduction and main results}
This paper studies the following question
\begin{equation} \label{eq1.1}
 \left\{\aligned
&-\Delta_{x,y}u+\lambda_1 u =|u|^{2}u+u|v|^2,\ (x, y) \in \mathbb{R}^3  \times \mathbb{T} , \\
&-\Delta_{x,y}v+\lambda_2 v =|v|^{2}v+v|u|^2,\ (x, y) \in \mathbb{R}^3  \times \mathbb{T} , \\
&\int_{\mathbb{R}^3  \times \mathbb{T} }u^2dxdy=\int_{\mathbb{R}^3  \times \mathbb{T} }v^2dxdy=\Theta^2,
\endaligned
\right.
\end{equation}
where the mass $\Theta>0$, the frequency $\lambda_1\cdot\lambda_2\geq0$. We are interested in the property of solution for \eqref{eq1.1}, such as existence, $y$-dependence.

\textbf{1.1 Motivation.} As described in \cite{ACWZ2012}(also see Theorem 1.6 in \cite{HLTL2026}),
the question
\begin{equation*}
 \left\{\aligned
&-\Delta u+\lambda_1 u =|u|^{2}u+u|v|^2,\ x \in \mathbb{R}^4  , \\
&-\Delta v+\lambda_2 v =|v|^{2}v+v|u|^2,\ x \in \mathbb{R}^4
\endaligned
\right.
\end{equation*}
has no positive least energy solution(positive normalized solution) in the Euclidean space when $\lambda_1\cdot\lambda_2>0$, this is the direct conclusion of using Pohozaev identity. It is natural to ask whether the same result holds for the coupled Schr\"odinger equations on $\mathbb{R}^3  \times \mathbb{T}$? In this paper, we give an negative answer to this question. Indeed, we may find a solution $(u,v,\lambda_1,\lambda_2)$ satisfies \eqref{eq1.1} which is positive and $\lambda_1+\lambda_2>0$. Furthermore, we prove the $y$-dependence of normalized solution by using the scaling argument.

\eqref{eq1.1} comes from the time-dependent coupled Schr\"odinger equations
\begin{equation*}
 \left\{\aligned
&i\frac{\partial \Phi_1}{\partial t}+\Delta \Phi_1+\mu_1\Phi_1^3+\nu\Phi_1\Phi_2^2=0,\ x \in \mathbb{R}^3\times\mathbb{T}, \\
&i\frac{\partial \Phi_2}{\partial t}+\Delta \Phi_2+\mu_2\Phi_2^3+\nu\Phi_2\Phi_1^2=0,\ x \in \mathbb{R}^3\times\mathbb{T},
\endaligned
\right.
\end{equation*}
where $\mu_1,\ \mu_2$ and $\nu$ describe the intraspecies and interspecies scattering lengths. The
sign of $\nu$ determines whether the interactions are repulsive or attractive, that is, the interaction is attractive when $\nu>0$ and repulsive when $\nu<0$. The existence of solutions has recently received great interest, see \cite{ACWZ2012} for fixed frequency issue. Another important feature of system is that
\begin{equation*}
  \int_{\mathbb{R}^3  \times \mathbb{T}}|\Phi_1(t,x,y)|^2dxdy=\int_{\mathbb{R}^3  \times \mathbb{T}}|\Phi_2(t,x,y)|^2dxdy=\Theta^2
\end{equation*}
of the wave functions are conserved. For the existence of normalized solution, we can refer to \cite{HLTL2026} for more detail.

Recently, there has been an increasing interest in studying dispersive equations on the waveguide manifolds $\mathbb{R}^d\times\mathbb{T}^n$ whose mixed type geometric nature makes the underlying analysis rather challenging and delicate. For the nonlinear Schr\"odinger model, there have been many results on waveguide manifolds, as shown in \cite{{XCZG2020},{ZHBP2014},{RKJM2021},{YL2023},{ZZH2021},{ZZJZ2021},{NTNV2012},{ADI2012},{XYHY2024}}. However, there is relatively little research for coupled Schr\"odinger equation on the waveguide manifolds. Terracini et. al in \cite{STNT2014}  consider the Cauchy problems
\begin{equation*}
  i\partial_tu-\Delta_{x,y} u-u|u|^\alpha=0,\ u|_{t=0}=u_0,
\end{equation*}
where $(x,y)\in\mathbb{R}_x^d\times M_y^k$, $M_y^k$ is a compact Riemannian manifold. They proved that above a critical mass the ground states have nontrivial $M_y^k$ dependence.  After that, Hajaiej et. al in \cite{HHYL2024} consider the following question
\begin{equation*}
  \Delta_{x,y}^2u-\beta \Delta_{x,y}u + \theta u = |u|^{\alpha}u, \ (x,y)\in\mathbb{R}^d\times\mathbb{T}^n ,
\end{equation*}
where $\beta \in \mathbb{R}$ and $\alpha \in (0, \frac{8}{d+n})$. They extend the conclusion in \cite{STNT2014} to four order Schr\"odinger equation by using new scaling argument.  Therefore, it is natural to ask whether the $y$-dependence holds for coupled Schr\"odinger equations? We give an affirmative answer in this paper.

The natural energy functional associated with the equation \eqref{eq1.1} is given by
\begin{equation*}
I(u,v)=\frac{1}{2} \int_{\mathbb{R}^3 \times\mathbb{T} }[| \nabla_{x,y}u|^2+ |\nabla_{x,y}v|^2 ] d xdy -\frac{1}{4} \int_{\mathbb{R}^3\times\mathbb{T} }(|u|^4+2|u|^2|v|^2+|v|^4)d xdy.
\end{equation*}
Normalized solutions to \eqref{eq1.1} correspond to critical points of the energy functional $I$ restricted on $S_\Theta\times S_\Theta$, where
\begin{equation*}
S_\Theta:= \{u\in H^1(\mathbb{R}^3 \times\mathbb{T}) : \|u\|_2 = \Theta > 0\}
\end{equation*}
and $m_\Theta$ is defined as
\begin{equation*}
  m_\Theta=\inf\{I(u,v): (u,v)\in V_\Theta\},
\end{equation*}
where
\begin{equation*}
  Q(u,v)=\int_{\mathbb{R}^3\times \mathbb{T} }(|\nabla_xu |^{2}+|\nabla_xv |^{2})dxdy -\frac{3}{4}\int_{\mathbb{R}^3\times\mathbb{T}}(|u|^{4}+2|u|^{2}|v|^{2}+|v|^{4})dxdy
\end{equation*}
and
\begin{equation*}
  V_\Theta=\{(u,v)\in S_\Theta\times S_\Theta,Q(u,v)=0\}.
\end{equation*}
In order to study the periodic dependency results, we define
\begin{equation*}
\widehat{S_\Theta}:= \{u\in H^1(\mathbb{R}^3) : \|u\|_2 = \Theta > 0\}.
\end{equation*}

\textbf{1.2 Main results.}  The main results of this paper are as follows.

\begin{theorem}\label{t1.1}
The problem \eqref{eq1.1}
has a couple $(u,v,\lambda_1,\lambda_2)$ solution for any $\Theta>0$, where $(u,v)$ is positive and $\lambda_1+\lambda_2>0$.
\end{theorem}
The proof of Theorem \ref{t1.1} is based on the semivirial-vanishing technology developed by Luo in \cite{{YL2024},{YL2022}}, and we extend it to the system.
\begin{remark}\label{R1.1}
(i) In this paper, we only consider the Sobolev critical case. For mass subcritical case, the existence of solution follows directly
from the boundedness of $I$ from below on $S_\Theta\times S_\Theta$, the Gagliardo-Nirenberg inequality on waveguide manifold and a standard minimization argument. For mass supercritical case, the proof method of Theorem \ref{t1.1} is valid.

(ii) Unlike the study of normalized solution the critical Schr\"odinger equation, this article does not use an energy threshold to exclude vanishing and bifurcation.

(iii) The conclusion of Theorem \ref{t1.1} can be extended to systems with any number of components, that is, for $k\geq2$,
\begin{equation*}
  -\Delta_{x,y}u_i+\lambda_iu_i=\sum_{j=1}^ku_j^2u_i, (x,y)\in\mathbb{R}^3\times\mathbb{T},\ i=1,\cdots,k
\end{equation*}
with the normalization condition
\begin{equation*}
  \int_{\mathbb{R}^3\times\mathbb{T}}u_i^2dxdy=\Theta_i^2,\  i=1,\cdots,k.
\end{equation*}

(iv) We point out that in \cite{HLTL2026}, as well as in all
the contributions related to the problem with fixed frequencies, one of the main difficulties is represented by the fact that one searches for solutions having both $u\neq0$ and $v\neq0$. Under the setting of waveguide manifold, this problem still exists. Moreover, we must prevent the loss of part of the mass of one of the components in the passage to the limit.

(v) As described in \cite{HLTL2026}, the equations
\begin{equation*}
 \left\{\aligned
&-\Delta u+\lambda_1 u =|u|^4u+\frac{\alpha}{6}|u|^{\alpha-2}v^{\beta}u,\ x \in \mathbb{R}^3  , \\
&-\Delta v+\lambda_2 v =|v|^4v+\frac{\beta}{6}|v|^{\beta-2}u^{\alpha}v,\ x \in \mathbb{R}^3
\endaligned
\right.
\end{equation*}
has no  positive normalized solution in the Euclidean space $\mathbb{R}^3$ when $\lambda_1\cdot\lambda_2>0$, where $\alpha>1,\beta>1,\alpha+\beta=6$. The proof method of Theorem \ref{t1.1} is also valid for this question.
\end{remark}
Borrowing the ideas from \cite{STNT2014} and new scaling techniques, we obtain the following periodic dependency results.
\begin{theorem}\label{t1.2}
 For any $\Theta \in (0,\infty )$, let $(u_\Theta,u_\Theta)$ be a minimizer
of $m_\Theta$ deduced from Theorem \ref{t1.1}. Then there exists some $\Theta_0\in(0,+\infty)$ such that

$\mathrm{(1)}$ for all  $\Theta \in(0,\Theta_0)$ we have $m_\Theta<2\pi\widehat{m}_{\frac{\Theta}{\sqrt{2\pi}}}$. Moreover, for $\Theta \in(0,\Theta_0)$ any minimizer $(u_\Theta,v_\Theta)$
of $m_\Theta$
satisfies $(u_{y,\Theta},v_{y,\Theta})\neq0$.

$\mathrm{(2)}$ for all  $\Theta \in(\Theta_0,+\infty)$ we have $m_\Theta=2\pi\widehat{m}_{\frac{\Theta}{\sqrt{2\pi}}}$. Moreover, for $\Theta \in(\Theta_0,+\infty)$ any minimizer $(u_\Theta,v_\Theta)$
of $m_\Theta$
satisfies $(u_{y,\Theta},v_{y,\Theta})= 0$.
\end{theorem}
\begin{remark}
(i)  The conclusion of Theorem \ref{t1.2} can be extended to systems with any number of components, that is, for $k\geq2$,
\begin{equation*}
  -\Delta_{x,y}u_i+\lambda_iu_i=\sum_{j=1}^ku_j^2u_i, (x,y)\in\mathbb{R}^3\times\mathbb{T},\ i=1,\cdots,k
\end{equation*}
with the normalization condition
\begin{equation*}
  \int_{\mathbb{R}^3\times\mathbb{T}}u_i^2dxdy=\Theta_i^2,\  i=1,\cdots,k.
\end{equation*}
Different from Theorem \ref{t1.1},  we require each component to have the same mass, that is, $\Theta_1^2=\cdots=\Theta_k^2$. An interesting question is whether there is still a periodic dependency conclusion when $\Theta_i^2\neq\Theta_j^2,\ i\neq j,\ i=1,\cdots,k$?

(ii) For $\mathbb{R}^d\times\mathbb{T}, d\geq3$, we may obtain similar results for \eqref{eq1.1}. If $d=2$, we can observe an interesting phenomenon, that is, \eqref{eq3.1} does not have scaling property.
\end{remark}
The structure of this paper is arranged as follows. In section 2, we study the existence of normalized solution. Next, we will study $y$-dependence of the solution.

\section{Existence of solution}
In this section, we are devoted to the existence of normalized solution.
In order to deal with the non-vanishing
weak limit of a minimizing sequence, we use a scale-invariant Gagliardo-Nirenberg inequality on $\mathbb{R}^3\times \mathbb{T}$, the proof can be found in  \cite{YL2022}.
\begin{proposition}\label{p2.1}
There exists some $C > 0$ such that for all $u\in X$ we have
 \begin{equation*}
     \int_{\mathbb{R}^3\times \mathbb{T} }|u|^{4}dxdy\leq C\|\nabla_xu\|_2^3 \|u\|_2^{2}(\|u\|_2+\|u_y\|_2).
 \end{equation*}
\end{proposition}
\begin{lemma}\label{L2.1}
For any $\Theta \in(0, \infty)$ we have $m_\Theta \in(0, \infty)$.
\end{lemma}
\begin{proof}
Assume $\left(u_n,v_n\right)_n \subset V_{\Theta}$ be a minimizing sequence such that $I\left(u_n,v_n\right)=m_\Theta+o_n(1)$ and $Q(u_n,v_n)=o_n(1)$, one has
\begin{eqnarray*}
m_\Theta+o_n(1)&=&I\left(u_n,v_n\right)-\frac{1}{3} Q\left(u_n,v_n\right)\\
&=&\frac{1}{2}\int_{\mathbb{R}^3 \times\mathbb{T} }(|u_{y,n}|^2+|v_{y,n}|^2)dxdy+\frac{1}{6} \int_{\mathbb{R}^3 \times\mathbb{T} }[| \nabla_{x}u_n|^2+ |\nabla_{x}v_n|^2 ] d xdy,
\end{eqnarray*}
which implies that $\left(u_n,v_n\right)_n$ is a bounded sequence in $H^1\times H^1$. Moreover, by using Proposition \ref{p2.1}, we obtain that
\begin{eqnarray*}
\int_{\mathbb{R}^3\times \mathbb{T} }(|\nabla_xu_n |^{2}+|\nabla_xv_n |^{2})dxdy &=&\frac{3}{4}\int_{\mathbb{R}^3\times\mathbb{T}}(|u_n|^{4}+2|u_n|^{2}|v_n|^{2}+|v_n|^{4})dxdy\\
&\leq&C\int_{\mathbb{R}^3\times\mathbb{T}}(|u_n|^{4}+|v_n|^{4})dxdy\\
&\leq&C\left(\int_{\mathbb{R}^3\times \mathbb{T} } |\nabla_xu_n |^{2}dxdy+\int_{\mathbb{R}^d\times \mathbb{T} } |\nabla_xv_n |^{2}dxdy\right)^\frac{3}{2},
\end{eqnarray*}
then
\begin{equation*}
  \liminf _{n \rightarrow \infty}\int_{\mathbb{R}^3\times\mathbb{T}}(|u_n|^{4}+|v_n|^{4})dxdy\sim\liminf _{n \rightarrow \infty}\int_{\mathbb{R}^3\times \mathbb{T} }(|\nabla_xu_n |^{2}+|\nabla_xv_n |^{2})dxdy.
\end{equation*}
Hence,
\begin{eqnarray*}
m_\Theta\geq\frac{1}{6}\int_{\mathbb{R}^3\times \mathbb{T} }(|\nabla_xu_n |^{2}+|\nabla_xv_n |^{2})dxdy\geq C>0,
\end{eqnarray*}
which completes the proof.
\end{proof}
Define $(u^\tau,v^\tau)=(\tau^{\frac{3}{2}}u(\tau x,y),\tau^{\frac{3}{2}}v(\tau x,y)), \tau\in(0, +\infty)$, then we prove some useful properties of $ Q\left(u^\tau\right)$ and $m_\Theta$.
\begin{lemma}\label{L2.2}
 Let $\Theta>0$ and $(u,v) \in S_\Theta\times S_\Theta$. Then the following statements hold true:

$\mathrm{(i)}$ $\frac{d}{d \tau} I\left(u^\tau,v^\tau\right)=\tau^{-1} Q\left(u^\tau,v^\tau\right)$ for all $\tau>0$.

$\mathrm{(ii)}$ There exists some $\tau^*=\tau^*(u,v)>0$ such that $(u^{\tau^*},u^{\tau^*}) \in V_\Theta$.

$\mathrm{(iii)}$ We have $\tau^*<1$ if and only if $Q(u,v)<0$. Moreover, $\tau^*=1$ if and only if $Q(u,v)=0$.

$\mathrm{(iv)}$ Following inequalities hold:
$$
Q\left(u^\tau,v^\tau\right)\left\{\begin{array}{l}
>0, \quad \tau \in\left(0, \tau^*\right), \\
<0, \quad \tau \in\left(\tau^*, \infty\right) .
\end{array}\right.
$$

$\mathrm{(v)}$ $I\left(u^\tau,v^\tau\right)<I\left(u^{\tau^*},v^{\tau^*}\right)$ for all $\tau>0$ with $\tau \neq \tau^*$.
\end{lemma}
\begin{proof}
By calculation, we have
\begin{eqnarray*}
  \frac{d}{d\tau}I\left(u^\tau,v^\tau\right)&=&\frac{d}{d\tau} \left[ \int_{\mathbb{R}^3\times \mathbb{T} }\left(\frac{\tau^2}{2}|  \nabla_x u |^{2}+\frac{1}{2}| u_{y}|^2+\frac{\tau^2}{2}|  \nabla_x v|^{2}+\frac{1 }{2}| v_{y}|^2\right)dxdy \right. \\
  &&-\left.\frac{\tau^{3}}{4}\int_{\mathbb{R}^3\times\mathbb{T}}(|u|^{4}+2|u|^2|v|^2+|v|^{4})dxdy \right] \\
  &=& \tau\int_{\mathbb{R}^3\times \mathbb{T} }(|\nabla_x u|^{2}+|\nabla_x v|^2)dxdy-\frac{3}{4}\tau^{2}\int_{\mathbb{R}^3\times\mathbb{T}}(|u|^{4}+2|u|^2|v|^2+|v|^{4})dxdy\\
  &=&\tau^{-1} Q\left(u^\tau,v^\tau\right).
\end{eqnarray*}
Now, let $g(\tau):=\frac{d}{d \tau} I\left(u^\tau,v^\tau\right)$, then $g(\tau)$ has a zero at $\tau^*>0$, $g(\tau)$ is positive on $\left(0, \tau^*\right)$ and  negative on $(\tau^*, \infty)$. (ii) and (iv) follow from
$$
g(\tau)=\frac{d I\left(u^\tau,v^\tau\right)}{d \tau}=\frac{Q\left(u^\tau,v^\tau\right)}{\tau} .
$$
In order to get (iii), let $Q(u,v)<0$, then
$$
g(\tau^*)=0>Q(u,v)=\frac{Q\left(u^1,v^1\right)}{1}=g(1),
$$
which implies $\tau^*<1$. Conversely, let $\tau^*>1$. Since $g(\tau)$ is monotone increasing on $\left(0, \infty\right)$ we obtain
$$
Q(u,v)=g(1)<g\left(\tau^*\right)=0.
$$
To see (v), integration by parts yields
$$
I\left(u^{\tau^*},v^{\tau^*}\right)=I\left(u^\tau,v^\tau\right)+\int_\tau^{\tau^*} g(s) d s.
$$
(v) follows from the fact that $g(\tau)$ is positive on $\left(0, \tau^*\right)$ and $y(\tau)$ is negative on $\left(\tau^*, \infty\right)$.
\end{proof}
\begin{lemma}\label{L2.3}
The mapping $\Theta\rightarrow m_\Theta$ is continuous and monotone decreasing on $(0, +\infty)$.
\end{lemma}
The proof of Lemma \ref{L2.3} follows from the same line in \cite{{JB2013},{YL2022}}. Now, we prove that any minimizer $(u_\Theta,v_\Theta)$ of $m_\Theta$ is automatically a solution of \eqref{eq1.1}.
\begin{definition}\label{D2.1}
We say that $I(u,v)$ has a mountain pass geometry on $S_\Theta\times S_\Theta$ at the level $\gamma_\Theta$ if there exists some $k>0$ and $\varepsilon \in\left(0, m_\Theta\right)$ such that
$$
\gamma_\Theta:=\inf _{g \in \Gamma(\Theta)} \max _{t \in[0,1]} I(g(t))>\max \left\{\sup _{g \in \Gamma(\Theta)} I(g(0)), \sup _{g \in \Gamma(\Theta)} I(g(1))\right\},
$$
where
$$
\Gamma(\Theta):=\left\{g \in C([0,1] ; S_\Theta\times S_\Theta): g(0)=(u_1,v_1) \in A_{k, \varepsilon}, I(g(1))=I(u_2,v_2)<0\right\}
$$
and
$$
A_{k, \varepsilon}:=\left\{(u,v) \in S_\Theta\times S_\Theta:\|\nabla_x u\|_2^2+\|\nabla_x v\|_2^2<k,\left\|  u_y\right\|_2^2+\left\|  v_y\right\|_2^2 \leq 2\left(m_\Theta-\varepsilon\right)\right\}.
$$
\end{definition}
\begin{lemma}\label{L2.4}
There exist $k>0$ and $\varepsilon \in\left(0, m_\Theta\right)$ such that

(i) $m_\Theta=\gamma_\Theta$.

(ii) $I(u,v)$ has a mountain pass geometry on $S_\Theta$ at the level $m_\Theta$ in the sense of Definition \ref{D2.1}.
\end{lemma}
\begin{proof}
First, we prove that $Q(u,v)>0$ for all $u \in A_{k, \varepsilon}$, where $k$ is independent of the choice of $\varepsilon \in\left(0, m_\Theta\right)$ and small enough. In fact, it follows from  $\left\|  u_y\right\|_2^2+\left\|  v_y\right\|_2^2<2 m_\Theta$ for $u \in A_{k, \varepsilon}$ that
\begin{eqnarray*}
Q(u,v)&=&\int_{\mathbb{R}^3\times \mathbb{T} }(|\nabla_xu |^{2}+|\nabla_xv |^{2})dxdy -\frac{3}{4}\int_{\mathbb{R}^3\times\mathbb{T}}(|u|^{4}+2|u|^{2}|v|^{2}+|v|^{4})dxdy\\
 &\geq&\int_{\mathbb{R}^3\times \mathbb{T} }(|\nabla_xu |^{2}+|\nabla_xv |^{2})dxdy-C\left(\int_{\mathbb{R}^3\times \mathbb{T} } (|\nabla_xu |^{2}+|\nabla_xv |^{2})dxdy\right)^\frac{3}{2}>0
\end{eqnarray*}
as long as $|\nabla_xu |^{2}+|\nabla_xv |^{2} \in(0, k)$ for some sufficiently small $k$. Next we construct the number $\varepsilon$. By using Lemma \ref{L2.1}, there exists $C_\Theta>0$ such that if $\left(u_n,v_n\right)_n \subset V_\Theta$ is a minimizing sequence for $m_\Theta$, then
\begin{equation}\label{eq2.1}
  \frac{1}{6}\int_{\mathbb{R}^3\times \mathbb{T} }(|\nabla_xu |^{2}+|\nabla_xv |^{2})dxdy \geq C_\Theta.
\end{equation}
In order to obtain $\beta<4 M_\Theta$, we must have
\begin{eqnarray}\label{eq2.2}
m_\Theta+\frac{C_\Theta}{4} &\geq&I\left(u_n,v_n\right)-\frac{1}{3} Q\left(u_n,v_n\right)\nonumber\\
&=&\frac{1}{2}\int_{\mathbb{R}^3 \times\mathbb{T} }(|u_{y,n}|^2+|v_{y,n}|^2)dxdy+\frac{1}{6} \int_{\mathbb{R}^3 \times\mathbb{T} }(| \nabla_{x,y}u_n|^2+ |\nabla_{x,y}v_n|^2 ) d xdy\nonumber\\
&\geq& \frac{1}{2}\int_{\mathbb{R}^3 \times\mathbb{T} }(|u_{y,n}|^2+|v_{y,n}|^2)dxdy+\frac{C_\Theta}{2}
\end{eqnarray}
for all sufficiently large $n$. Set $\varepsilon=\frac{C_\Theta}{4}$, for $u \in A_{k, \varepsilon}$, we have
\begin{eqnarray*}
I(u,v) &\leq& \frac{1}{2} \int_{\mathbb{R}^3 \times\mathbb{T} }[| \nabla_{x,y}u|^2+ |\nabla_{x,y}v|^2 ] d xdy -\frac{1}{4} \int_{\mathbb{R}^3\times\mathbb{T} }(|u|^4+2|u|^2|v|^2+|v|^4)d xdy\\
&\leq& m_\Theta-\varepsilon+\frac{1}{2} k+C k^{\frac{3}{2}}.
\end{eqnarray*}
Choosing $k=k(\varepsilon)$ sufficiently small such that $\frac{1}{2} k+C k^{\frac{3}{2}}<\frac{\varepsilon}{2}$, we have $I(u,v)<m_\Theta$ for all $u \in A_{k, \varepsilon}$ and by definition, (ii) follows immediately from (i).

Now, we prove $\gamma_\Theta \leq m_\Theta$. Let $\left(u_n,v_n\right)_n$ be the given minimizing sequence satisfying \eqref{eq2.1} and $(u,v)=\lim\limits_{n\rightarrow\infty}(u_n,v_n)$. For any $\kappa \in(0, \frac{C_\Theta}{4})$ we can choose $n$ sufficiently large such that
\begin{equation*}
  I(u,v) \leq m_\Theta+\kappa  \ \text{and} \ \frac{1}{6}\int_{\mathbb{R}^3\times \mathbb{T} }(|\nabla_xu |^{2}+|\nabla_xv |^{2})dxdy\geq C_\Theta.
\end{equation*}
Note that $\int_{\mathbb{R}^3 \times\mathbb{T} }(|u_{y}|^2+|v_{y}|^2)dxdy \leq 2\left(m_c-\varepsilon\right)$ for all $\kappa \in(0, \frac{C_\Theta}{4})$ by \eqref{eq2.2}. Let $(u^\tau,u^\tau)=(\tau^{\frac{3}{2}}u(\tau x,y),\tau^{\frac{3}{2}}v(\tau x,y))$, then $\left\|u_y^\tau\right\|_2^2=\left\|u_y\right\|_2^2$ for all $\tau \in(0, \infty)$ and $\left\|\nabla_xu^\tau\right\|_2^2=\tau^2\left\|\nabla_x u\right\|_2^2 \rightarrow 0$ as $\tau \rightarrow 0$. We then fix some $\tau_0>0$ sufficiently small such that $\left\|\nabla_x\left(u^{\tau_0}\right)\right\|_2^2+\left\|\nabla_x\left(v^{\tau_0}\right)\right\|_2^2<k$, which implies that $\left(u^{\tau_0},v^{\tau_0}\right) \in A_{k, \varepsilon}$. Moreover,
\begin{eqnarray*}
I\left(u^\tau,v^{\tau}\right)&=&\frac{\tau^2}{2} (\| \nabla_{x}u\|^2+ \|\nabla_{x}v\|^2)+\frac{1}{2} (\| u_y\|^2+ \|v_y\|^2) -\frac{\tau^3}{4} \int_{\mathbb{R}^3\times\mathbb{T} }(|u|^2+|v|^2)^2d xdy\\
&&\rightarrow-\infty
\end{eqnarray*}
as $\tau \rightarrow \infty$. Fix some $\tau_1$ sufficiently large such that $I\left(u^\tau,v^\tau\right)<0$. Now define $g \in C([0,1] ; S_\Theta\times S_\Theta)$ by
$$
g(t):=u^{\tau_0+\left(\tau_1-\tau_0\right) t}
$$
Then $g \in \Gamma(\Theta)$. By definition of $\gamma_\Theta$ and Lemma \ref{L2.2} we have
$$
\gamma_\Theta \leq \max\limits_{t \in[0,1]} I(g(t))=I(u,v) \leq m_c+\kappa.
$$

Finally, we claim that $\gamma_\Theta \geq m_\Theta$. Indeed, for any $g \in \Gamma(\Theta)$ we have $Q(g(0))>0$ because of $k$ small enough. We now prove $Q(g(1))<0$ for any $g \in \Gamma(\Theta)$. Assume the contrary that there exists some $g \in \Gamma(\Theta)$ such that $Q(g(1)) \geq 0$, we have
$$
0>I(g(1)) \geq I(g(1))-\frac{1}{3}Q(g(1))\geq 0,
$$
which is a contradiction. Next, by continuity of $g$ there exists some $\tau \in(0,1)$ such that $Q(g(\tau))=0$. Hence,
$$
\max\limits_{\tau \in[0,1]} I(g(\tau)) \geq m_\Theta,
$$
which completes the desired proof.
\end{proof}
\begin{lemma}\label{L2.5}
For any $\Theta > 0$ an optimizer $(u_\Theta,v_\Theta)$ of $m_\Theta$ is a solution of \eqref{eq1.1} for some $(\lambda_1,\lambda_2)\in \mathbb{R}\times\mathbb{R}$.
\end{lemma}
Now, we give a characterization for $m_\Theta$, which is more
useful for our analysis. Define
\begin{equation*}
  \widetilde{m_\Theta}:=\inf\{I(u,v):(u,v)\in S_\Theta\times S_\Theta,Q(u,v)\leq 0\}.
\end{equation*}
We claim that $m_\Theta=\widetilde{m_\Theta}$. In fact, obviously,
\begin{equation*}
  \widetilde{m_\Theta}\leq m_\Theta.
\end{equation*}
To obtain $\widetilde{m_\Theta}\geq m_\Theta$, let $(u_n,v_n)_n\subset S_\Theta\times S_\Theta$ be a
minimizing sequence such that
\begin{equation*}
  I(u_n,v_n)=\widetilde{m_\Theta}+o_n(1),\ Q(u_n,v_n)\leq 0,\ \forall n\in \mathbb{N}.
\end{equation*}
By Lemma \ref{L2.2} we know that there exists some $\tau_n\in(0, 1]$ such that $Q(u^{\tau_n}_n,v^{\tau_n}_n )=0$. Then
\begin{equation*}
  m_\Theta\leq I(u^{\tau_n}_n,v^{\tau_n}_n)\leq I(u^{\tau_n}_n,v^{\tau_n}_n)-\frac{1}{3}Q(u^{\tau_n}_n,v^{\tau_n}_n )\leq I(u_n,v_n)-\frac{1}{3}Q(u_n,v_n )=\widetilde{m_\Theta}.
\end{equation*}
Sending $n\rightarrow\infty$, we know $\widetilde{m_\Theta}\geq m_\Theta$.

\noindent\textbf{Proof of Theorem \ref{t1.1}} Firstly, we prove the existence of a non-negative optimizer of $m_\Theta$. Indeed, let $(u_n,v_n)_n\subset S_\Theta\times S_\Theta$ be a
minimizing sequence such that
\begin{equation*}
  I(u_n,v_n)=\widetilde{m_\Theta}+o_n(1),\ Q(u_n,v_n)\leq 0,\ \forall n\in \mathbb{N}.
\end{equation*}
By diamagnetic inequality we
know that $\widetilde{m_\Theta}$ is stable under the mapping $(u,v)\rightarrow(|u|,|v|)$, then it is not restrictive to suppose that $u_n\geq0, v_n\geq0$. Based on the above analysis and Lemma \ref{L2.1},
\begin{eqnarray*}
  \infty>\widetilde{m_\Theta}+o_n(1)&=&m_\Theta+o_n(1)\\
  &=&I(u_n,v_n)-\frac{1}{3}Q(u_n,v_n )\\
  &=&\frac{1}{2}\int_{\mathbb{R}^3 \times\mathbb{T} }(|u_{y,n}|^2+|v_{y,n}|^2)dxdy+\frac{1}{6} \int_{\mathbb{R}^3 \times\mathbb{T} }(| \nabla_{x}u_n|^2+ |\nabla_{x}v_n|^2 ) d xdy,
\end{eqnarray*}
which implies that $\left(u_n,v_n\right)_n$ is a bounded sequence in $H^1\times H^1$. Hence, there exists $(u,v)$ such that $(u_n,v_n) \rightharpoonup (u,v)$ weakly in $H^1\times H^1$. Note that $u=v\neq0$ since $\widetilde{m_\Theta}>0$ by using Lemma \ref{L2.1}. Now, we claim that \eqref{eq1.1} does not have semi-trivial solution, that is, $(0,v)$ or $(u,0)$. Without loss of generality, we assume that $v=0$ and $u\neq0$, then \eqref{eq1.1} become
\begin{equation}\label{eq2.3}
-\Delta_{x,y}u+\lambda_1u=|u|^{2}u,\ \|u\|_2=\Theta.
\end{equation}
By \cite{NSJFA2020},  any positive $C^2$-solution of \eqref{eq2.3} on $\mathbb{R}^4$ must be of the form
$$
\frac{b}{1+a^2|(x, y)-(x_0, y_0)|^2}
$$
with some $a, b>0$ and $\left(x_0, y_0\right) \in \mathbb{R}^{4}$. However, in this case $u$ can not be periodic along the $y$-direction, which leads to a contradiction.

By weakly lower semicontinuity of norms, if $\|u\|_2= \|v\|_2$, then
$$
\|u\|_2= \|v\|_2=: \Theta_1 \in(0, \Theta], \  I(u,v) \leq \widetilde{m_\Theta}.
$$
We next show $Q(u,v) \leq 0$. Assume the contrary $Q(u,v)>0$. By Brezis-Lieb lemma, $Q\left(u_n,v_n\right) \leq 0$, one has
\begin{equation*}
  \|u_n-u\|_2^2 =\|v_n-v\|_2^2  =\Theta^2-\Theta_1^2+o_n(1),
\end{equation*}
and
\begin{eqnarray*}
Q\left(u_n-u,v_n-v\right)\leq-Q(u,v)+o_n(1)<0.
\end{eqnarray*}
There exists some $\tau_n \in(0,1)$ such that $Q\left(\left(u_n-u\right)^{\tau_n},\left(v_n-v\right)^{\tau_n}\right)=0$. Consequently,
\begin{eqnarray*}
\widetilde{m_\Theta} &\leq& I\left(\left(u_n-u\right)^{\tau_n},(v_n-v)^{\tau_n}\right)-\frac{1}{3}Q\left(\left(u_n-u\right)^{\tau_n},\left(v_n-v\right)^{\tau_n}\right)\\
&<&I\left(u_n-u,v_n-v\right)=I\left(u_n,v_n\right)-I(u,v)+o_n(1)=\widetilde{m_\Theta}-I(u,v)+o_n(1) .
\end{eqnarray*}
Taking $n \rightarrow \infty$, we obtain $I(u,v)=0$. This implies $(u,v)=0$, which is a contradiction, so $Q(u,v) \leq 0$. If $Q(u,v)<0$, then again by Lemma \ref{L2.1} we find some $s \in(0,1)$ such that $Q\left(u^s,v^s\right)=0$. But then using Lemma \ref{L2.2} and $\Theta_1 \leq \Theta$,
$$
\widetilde{m_{\Theta_1}} \leq I\left(u^s,v^s\right)-\frac{1}{3}Q\left(u^s,v^s\right)<I(u,v)-\frac{1}{3}Q(u,v) \leq \widetilde{m_\Theta} \leq \widetilde{m_{\Theta_1}},
$$
which is a contradiction. If $\|u\|_2\neq \|v\|_2$,  we may assume $\|u\|_2=\Theta_1<\|v\|_2=\Theta_2\leq\Theta$, then $I(u,v) \leq \widetilde{m_{\Theta_1,\Theta_2}}$, where
 \begin{equation*}
   \widetilde{m_{\Theta_1,\Theta_2}}:=\inf\{I(u,v): S_{\Theta_1}\times S_{\Theta_2},Q(u,v)\leq0\}.
 \end{equation*}
By Brezis-Lieb lemma, we have
\begin{equation*}
  \|u_n-u\|_2^2 =\Theta^2-\Theta_1^2+o_n(1),\ \|v_n-v\|_2^2  =\Theta^2-\Theta_2^2+o_n(1).
\end{equation*}
Similar to the previous situation, we also have $Q\left(u^s,v^s\right)=0$, $s\in(0,1)$.
Then
$$
\widetilde{m_{\Theta_1,\Theta_2}} \leq I\left(u^s,v^s\right)-\frac{1}{3}Q\left(u^s,v^s\right)<I(u,v)-\frac{1}{3}Q(u,v)  \leq \widetilde{m_{\Theta_1,\Theta_2}}.
$$
which is absurd. Therefore, $Q(u,v)=0$, then $(u,v)$ is a minimizer of $m_{\Theta_1}$. From Lemma \ref{L2.5} we know that $(u,v)$ is a solution of \eqref{eq1.1}.

Next, we prove that $\lambda_1+\lambda_2$ is non-negative. Note that
\begin{equation*}
   \int_{\mathbb{R}^3 \times\mathbb{T} } | \nabla_{x,y}u|^2 d xdy+\lambda_1\Theta^2  =\int_{\mathbb{R}^3\times\mathbb{T} }(|u|^4+|u|^2|v|^2)d xdy,
\end{equation*}
\begin{equation*}
   \int_{\mathbb{R}^3 \times\mathbb{T} } | \nabla_{x,y}v|^2 d xdy+\lambda_2\Theta^2  =\int_{\mathbb{R}^3\times\mathbb{T} }(|v|^4+|u|^2|v|^2)d xdy,
\end{equation*}
and
\begin{equation*}
  \int_{\mathbb{R}^3\times \mathbb{T} }(|\nabla_xu |^{2}+|\nabla_xv |^{2})dxdy =\frac{3}{4}\int_{\mathbb{R}^3\times\mathbb{T}}(|u|^{4}+2|u|^{2}|v|^{2}+|v|^{4})dxdy
\end{equation*}
then
$$
\left\|u_y\right\|_2^2+\left\|v_y\right\|_2^2+(\lambda_1+\lambda_2)\Theta^2=\frac{1}{4}\int_{\mathbb{R}^3\times\mathbb{T}}(|u|^{4}+2|u|^{2}|v|^{2}+|v|^{4})dxdy.
$$
Define $T_\tau(u,v)=(\tau^{-2} u(\tau^{-2} x,y),\tau ^{-2}v(\tau^{-2} x,y))$. By using Lemma \ref{L2.3} and $(u,v)$  is an optimizer of $m_\Theta$, we know $\frac{d}{d\tau}I(T_\tau(u,v))\left|_{\lambda=1}\right.\leq0$, that is,
\begin{equation*}
   -\|\nabla_xu\|_2^2 -\|\nabla_xv\|_2^2+\left\|u_y\right\|_2^2 +\left\|v_y\right\|_2^2+\frac{1}{2}\int_{\mathbb{R}^3\times\mathbb{T}}(|u|^{4}+2|u|^{2}|v|^{2}+|v|^{4})dxdy\leq0. \end{equation*}
By using $Q(u,v)=0$, then
\begin{equation*}
  \left\|u_y\right\|_2^2+\left\|v_y\right\|_2^2\leq\frac{1}{4}\int_{\mathbb{R}^3\times\mathbb{T}}(|u|^{4}+2|u|^{2}|v|^{2}+|v|^{4})dxdy.
\end{equation*}
Hence, we infer that $(\lambda_1+\lambda_2) \Theta^2 \geq 0$. Since $\Theta> 0$ we conclude $(\lambda_1+\lambda_2) \geq 0$. If $\lambda_1+\lambda_2=0$, then we get $\lambda_1=\lambda_2=0$ or $\lambda_1+\lambda_2>0$ because of $\lambda_1\cdot\lambda_2\geq0$.

\textbf{Case 1:  $\lambda_1=\lambda_2=0$.}  Assume that $(u,v)$ satisfies the equation
\begin{equation*}
 \left\{\aligned
&-\Delta_{x,y}u=|u|^{2}u+u|v|^2,\ (x, y) \in \mathbb{R}^3  \times \mathbb{T} , \\
&-\Delta_{x,y}v=|v|^{2}v+v|u|^2,\ (x, y) \in \mathbb{R}^3  \times \mathbb{T} .
\endaligned
\right.
\end{equation*}
By using the elliptic regularity theory, we infer that $(u,v) \in C^2(\mathbb{R}^{4})$. Using strong maximum principle we also know that $(u,v)$ is positive. By \cite{ACWZ2012}, any positive $C^2$-solution of \eqref{eq1.1} on $\mathbb{R}^4$ must be of the form
$$
(\frac{b}{1+a^2|(x, y)-(x_0, y_0)|^2},\frac{d}{1+c^2|(x, y)-(x_0, y_0)|^2})
$$
with some $a, b, c, d>0$ and $\left(x_0, y_0\right) \in \mathbb{R}^{4}$. However, in this case $u$ can not be periodic along the $y$-direction, which leads to a contradiction.

\textbf{Case 2:  $\lambda_1+\lambda_2>0$.} We prove $\|u\|_2=\|v\|_2=\Theta$. Assume  $\Theta_1<\Theta$, by Lemma \ref{L2.3}, we know that $m_{\Theta_1}$ is a local minimizer of the mapping $\Theta \mapsto m_\Theta$, which implies
\begin{equation*}
  \left\|u_y\right\|_2^2+\left\|v_y\right\|_2^2=\frac{1}{4}\int_{\mathbb{R}^3\times\mathbb{T}}(|u|^{4}+2|u|^{2}|v|^{2}+|v|^{4})dxdy.
\end{equation*}
then $(\lambda_1+\lambda_2)\|u\|_2=0$, which is a contradiction because $\lambda_1+\lambda_2>0$ and $(u,v) \neq 0$. Hence $\|u\|_2=\|v\|_2=\Theta$. By using strong maximum principle, we get that $(u,v)$ is positive.

\section{Dependency of the ground states}
In order to study the $y$-dependence of the ground states, we need some notation on the Euclidean space $\mathbb{R}^3$. Define
\begin{equation*}
  \widehat{m}_{\rho}=\inf_{\substack{\|u\|_2=\|v\|_2=\rho,\\ (u,v)\in H^1\times H^1}}\widehat{I}(u,v),
\end{equation*}
where
\begin{equation}\label{eq3.1}
  \widehat{I}(u,v)=\frac{1}{2}\int_{\mathbb{R}^3}|\nabla_x u|^2dx +\frac{1}{2}\int_{\mathbb{R}^3}|\nabla_x v|^2dx-\frac{1}{4}\int_{\mathbb{R}^3}(|u|^4+2|u|^2|v|^2+|v|^4)dx .
\end{equation}
Let $(\widehat{u},\widehat{v})=(\rho^{-2}u(\rho^{-2}x),\rho^{-2}v(\rho^{-2}x))$, we have
\begin{equation*}
  \widehat{m}_{\rho}=\rho^{-2}\widehat{m}_1.
\end{equation*}
It is well known that $+\infty>\widehat{m}_\rho>0$ for all $\rho>0$. On the one hand, assume that $(u, v)$ is the critical point to \eqref{eq3.1}, we have the following Pohozaev type identity
\begin{equation*}
    \int_{\mathbb{R}^3}|\nabla_xu|^2dx+\int_{\mathbb{R}^3}|\nabla_xv|^2dx=\frac{3}{4}\int_{\mathbb{R}^3}|u|^4dx+\frac{3}{4}\int_{\mathbb{R}^3}|v|^4dx+\frac{3}{2}\int_{\mathbb{R}^3}|u|^{2}|v|^{2}dx.
\end{equation*}
On the other hand,
\begin{equation*}
    \int_{\mathbb{R}^3}|\nabla_xu|^2dx+\omega_1\int_{\mathbb{R}^3}|u|^2dx= \int_{\mathbb{R}^3}|u|^4dx+\int_{\mathbb{R}^3}|u|^{2}|v|^{2}dx
\end{equation*}
and
\begin{equation*}
   \int_{\mathbb{R}^3}|\nabla_xv|^2dx+\omega_2\int_{\mathbb{R}^3}|v|^2dx= \int_{\mathbb{R}^3}|v|^4dx+\int_{\mathbb{R}^3}|u|^{2}|v|^{2}dx.
\end{equation*}
Therefore,
\begin{eqnarray*}
&&\omega_1\int_{\mathbb{R}^3}|u|^2dx+\omega_2\int_{\mathbb{R}^3}|v|^2dx\\
&=&\frac{1}{3} \int_{\mathbb{R}^3}|\nabla_xu|^2dx+\frac{1}{3} \int_{\mathbb{R}^3}|\nabla_xv|^2dx\\
&=&2\left(\frac{1}{2}\int_{\mathbb{R}^3}|\nabla_x u|^2dx-\frac{1}{4}\int_{\mathbb{R}^3}|u|^4dx+\frac{1}{2}\int_{\mathbb{R}^3}|\nabla_x v|^2dx-\frac{1}{4}\int_{\mathbb{R}^3}|v|^4dx-\frac{1}{2}\int_{\mathbb{R}^3}|u|^{2}|v|^{2}dx\right)\\
&=&2\widehat{m}_{\|u\|_2}
\end{eqnarray*}
and
\begin{equation*}
  \widehat{m}_{\|u\|_2}= \widehat{I}(u)=\frac{1}{6}\left( \int_{\mathbb{R}^3}|\nabla_xu|^2dx+\int_{\mathbb{R}^3}|\nabla_xv|^2dx\right).
\end{equation*}
Next, we consider an auxiliary problem, that is, the minimization problems
\begin{equation*}
   m_{1,\mu}=\inf_{\substack{\|u\|_{L_{x,y}^2}=\|v\|_{L_{x,y}^2}=1,\\ (u,v)\in H^1\times H^1}}J_\mu(u,v),
\end{equation*}
where
\begin{equation*}
  J_\mu(u,v)=\frac{1}{2}\int_{\mathbb{R}^3\times \mathbb{T} }\left(|  \nabla_xu|^{2}+\mu|u_{y}|^{2}+|  \nabla_xv|^{2}+\mu|v_{y}|^{2}\right)dxdy-\frac{1}{4}\int_{\mathbb{R}^3\times\mathbb{T}}(|u|^{4}+2|u|^2|v|^2+|v|^{4})dxdy .
\end{equation*}
\subsection{Parameter tends to zero}

Now, we consider the case $\mu\rightarrow0$.
\begin{lemma}\label{L3.1}
$\lim\limits_{\mu\rightarrow0}m_{1,\mu}< 2\pi \widehat{m}_{\frac{1}{ \sqrt{2\pi} }}$.
\end{lemma}
\begin{proof}
Define the function $\varphi:[0,2\pi]\rightarrow[0,+\infty)$ by
 \begin{equation*}
   \varphi(y)=\left\{\aligned
    &0, \  y\in[0,a]\cup[2\pi-a,2\pi], \\
    &b(y-a), \  y\in[a,\pi],  \\
    &\varphi(2\pi-y),\   y\in[\pi,2\pi-a],
\endaligned
\right.
\end{equation*}
where $a \in (0, \pi )$ and $b \in (0,\infty )$. Now, we determine the value of $b$. By using direct calculations, it follows that
$$
\|\varphi\|_{L_{y}^2}^2=\frac{2 b^2(\pi-a)^3}{3},\ \|\varphi\|_{L_{y}^4}^4=\frac{2b^4(\pi-a)^{5}}{5}.
$$
Let $b=\left[\frac{5(\pi-a)^{-2}}{3}\right]^{\frac{1}{2}}$, we get $\|\varphi\|_{L_{y}^2}^2=\|\varphi\|_{L_{y}^4}^4$.
Hence,
$$
\|\varphi\|_{L_{y}^2}^2=\frac{2  (\pi-a)^3}{3} \cdot \left[\frac{5(\pi-a)^{-2}}{3}\right]=\frac{10(\pi-a)}{9}  .
$$
Let $\varphi_\varepsilon$ be the $\varepsilon$-mollifier of $\varphi$ on $[0,2 \pi]$ with some to be determined small $\varepsilon>0$. In particular, since $\varphi$ has compact support in $(0,2 \pi)$, so is $\varphi_{\varepsilon}$ for $\varepsilon \ll 1$. Next, let $(Q,Q)$ be an optimizer of $\widehat{m}_{\frac{1}{\|\varphi\|_{L_y^2} }}$ and define
$$
\psi(x, y):=Q(x) \cdot\frac{\|\varphi\|_{L_y^2}}{\left\|\varphi_\varepsilon\right\|_{L_y^2}}  \varphi_\varepsilon\left(y\right).
$$
This is to be understood that we extend $\varphi_\varepsilon$ $2 \pi$-periodically along the $y$-direction, which is possible since $\varphi_\varepsilon$ has compact support in $(0,2 \pi)$ when $\varepsilon \ll 1$. We have then $\|\psi\|_{L_{x,y}^2}=1$. Moreover,
\begin{eqnarray*}
J_0(\psi,\psi)&=&\int_{\mathbb{R}^3\times \mathbb{T} } |\nabla_x\psi|^{2} dxdy-\int_{\mathbb{R}^3\times\mathbb{T}}|\psi|^{4}dxdy\\
&= &\|\varphi\|_{L_y^2}^{2 } \| \nabla_xQ \|_{L_x^2}^2  - \|\varphi\|_{L_y^2}^{4}\|\varphi_\varepsilon\|_{L_y^2}^{-4} \| Q  \|_{L_x^4}^4 \|\varphi_\varepsilon\|_{L_y^4}^{4}\\
&= &\|\varphi\|_{L_y^2}^{2 }\left(\|\nabla_xQ \|_{L_x^2}^2-\| Q  \|_{L_x^4}^4\right)-\|\varphi\|_{L_y^2}^{2 }\left(\|\varphi\|_{L_y^2}^{2}\|\varphi_\varepsilon\|_{L_y^2}^{-4}  \|\varphi_\varepsilon\|_{L_y^4}^{4}-1\right)\| Q  \|_{L_x^4}^4  \\
&= & \|\varphi\|_{L_y^2}^{2 }\widehat{m}_{\frac{1}{\|\varphi\|_{L_y^2}}}+I.
\end{eqnarray*}
By using Lemma \ref{L2.4}, we can obtain that the mapping $\Theta \mapsto \Theta^{-2} \widehat{m}_\Theta$ is strictly decreasing on $(0, \infty)$. Note that $\|\varphi\|_{L_y^2} \rightarrow 0$ as $a \rightarrow \pi$, we can choose $a$ sufficiently close to $\pi$ such that
$$
\|\varphi\|_{L_y^2}^{2 }\widehat{m}_{\frac{1}{\|\varphi\|_{L_y^2} }}< 2 \pi   \widehat{m}_{\frac{1}{\sqrt{2\pi }}}.
$$
Therefore, there exists some $\zeta>0$ such that $\|\varphi\|_{L_y^2}^{2 }\widehat{m}_{\frac{1}{\|\varphi\|_{L_y^2} }}+\zeta<2 \pi   \widehat{m}_{\frac{1}{\sqrt{2\pi}}}$. By the properties of a mollifier operator, we also know that
$$
\left\|\varphi_\varepsilon\right\|_{L_y^2}^2=\|\varphi\|_{L_y^2}^2+o_\varepsilon(1)=\|\varphi\|_{L_y^4}^4+o_\varepsilon(1)=\|\varphi_\varepsilon\|_{L_y^4}^4+o_\varepsilon(1).
$$
Taking $\varepsilon$ sufficiently small, it follows  that $|I| \leq \zeta$, then we get $J_0(\psi,\psi)<2 \pi \widehat{m}_{\frac{1}{ \sqrt{2 \pi} }}$. Consequently,
$$
\lim\limits_{\mu \rightarrow 0} m_{1, \mu} \leq \lim \limits_{\mu \rightarrow 0} J_\mu(\psi,\psi)=J_0(\psi,\psi)<2 \pi \widehat{m}_{\frac{1}{ \sqrt{2 \pi }}},
$$
which completes the proof.
\end{proof}
\subsection{Parameter tends to infinity}
Now, we consider the case $\mu\rightarrow+\infty$.
\begin{lemma}\label{L3.2}
Let $u_\mu$ be an optimizer of $m_{1, \mu}$. Then we have
\begin{equation}\label{eq3.2}
  \lim\limits_{\mu \rightarrow \infty} m_{1, \mu}= 2 \pi  \widehat{m}_{\frac{1}{ \sqrt{2\pi} }}
\end{equation}
and
\begin{equation}\label{eq3.3}
  \lim\limits_{\mu \rightarrow \infty} \mu \int_{\mathbb{R}^3\times \mathbb{T} }(|u_y|^{2}+|v_y|^{2} )dxdy=0.
\end{equation}
\end{lemma}
\begin{proof}
By assuming a candidate $u \in S_1$ is independent of $y$ we can obtain that $m_{1, \mu} \leq2 \pi  \widehat{m}_{\frac{1}{ \sqrt{2\pi }}}$. Indeed, let $w(x)\in H^1(\mathbb{R}_x^3 )$ be such that $\|w\|_{L_x^2}^2=\frac{1}{ 2\pi }$ and $J_\mu(w)=\widehat{m}_{\frac{1}{\sqrt{2\pi}}}$. It follows that
\begin{eqnarray*}
  m_{1, \mu}  \leq J_\mu(w(x),w(x))
    = 2\pi  \widehat{m}_{\frac{1}{\sqrt{2\pi}}}.
\end{eqnarray*}
Now we prove that
\begin{equation}\label{eq3.4}
  \lim_{j\rightarrow\infty}\int_{\mathbb{R}^3\times\mathbb{T}}(|u_{y,\mu_j}|^{2}+|v_{y,\mu_j}|^{2})dxdy=0.
\end{equation}
In fact, suppose \eqref{eq3.4} does not hold, then up to a subsequence of $\mu_j$ we may assume that
$$
\inf\limits_{\mu>0}\int_{\mathbb{R}^3\times\mathbb{T}}(|u_{y,\mu_j}|^{2}+|v_{y,\mu_j}|^{2})dxdy=\zeta>0,\ \lim\limits_{j\rightarrow\infty}\mu_j=\infty,
$$
then
$$
\lim _{j\rightarrow \infty}(\mu_j-1)\int_{\mathbb{R}^3\times\mathbb{T}}(|u_{y,\mu_j}|^{2}+|v_{y,\mu_j}|^{2})dxdy=\infty.
$$
Note that $Q(u_{\mu_j},v_{\mu_j})=0$, the
\begin{eqnarray*}
  m_{1,\mu_j} &=&I(u_{\mu_j},v_{\mu_j})-\frac{1}{3}Q(u_{\mu_j},v_{\mu_j})   \\
    &=& \frac{2}{3}\int_{\mathbb{R}^3\times \mathbb{T} }\left(|  \nabla_xu_{\mu_j}|^{2} +|  \nabla_xv_{\mu_j}|^{2} \right)dxdy+\mu_j\int_{\mathbb{R}^3\times \mathbb{T} }\left( |u_{y,\mu_j}|^{2} + |v_{y.\mu_j}|^{2}\right)dxdy .  \\
    &\geq&\mu_j\int_{\mathbb{R}^3\times \mathbb{T} }\left( |u_{y,\mu_j}|^{2} + |v_{y,\mu_j}|^{2}\right)dxdy\rightarrow+\infty
\end{eqnarray*}
as $j \rightarrow \infty$. Now taking $j \rightarrow \infty$ we infer the contradiction $ 2 \pi  \widehat{m}_{\frac{1}{\sqrt{2 \pi}}} \geq m_{1, \mu} \rightarrow \infty$.

Now, we prove that $\liminf\limits_{j\rightarrow\infty} m_{1, \mu_j}\geq2 \pi  \widehat{m}_{\frac{1}{\sqrt{2 \pi}}}$. In fact, we introduce the functions $w_j(y)=\|u_{\mu_j}(x,y)\|_{L_x^2}^2=\|v_{\mu_j}(x,y)\|_{L_x^2}^2$, then $\|w_j(y)\|_{L_y^1}=1$. By using the product rule, it follows that
 \begin{eqnarray*}
  \int_{\mathbb{T}}| w_{y,j}(y)|dy
  &\leq&2\int_{\mathbb{R}^3\times \mathbb{T}} |u_{\mu_j}(x,y)u_{y,\mu_j}(x,y)| dxdy\\
  &\leq&C\|u_{\mu_j}\|_{L_{x,y}^2}\|u_{y,\mu_j}(x,y)\|_{L_{x,y}^2}.
 \end{eqnarray*}
Using \eqref{eq3.4}, we obtain
\begin{equation*}
  \lim\limits_{j\rightarrow\infty}\int_{\mathbb{T}}|w_{y,j}(y)|dy=0.
\end{equation*}
Combining this, $W^{1,1}\hookrightarrow L^{\infty}$ and  $\|w_j(y)\|_{L_y^1}=1$ with the Rellich compactness theorem, we know
\begin{equation}\label{eq3.5}
  \lim\limits_{j\rightarrow\infty}\left\|w_j(y)-\frac{1}{2\pi}\right\|_{L_y^r}=0,\ 1\leq r<\infty.
\end{equation}
Now, by the definition of $\widehat{m}_\rho$,
\begin{eqnarray*}
  &&\frac{1}{2}\int_{\mathbb{R}^3}(|\nabla_xu_{\mu_j}|^2+|\nabla_xv_{\mu_j}|^2)dx-\frac{1}{4}\int_{\mathbb{R}^3}(|u_{ \mu_j}|^4+2|u_{ \mu_j}|^2|v_{ \mu_j}|^2+|v_{ \mu_j}|^4)dx \\
  &\geq&\widehat{m}_{\|u_{\mu_j}(\cdot,y)\|_{L_x^2}}
    =\|u_{\mu_j}(\cdot,y)\|_{L_x^2}^{2}\widehat{m}_1
  = w_j \widehat{m}_1
\end{eqnarray*}
for all $y\in\mathbb{T}$. Therefore, by using \eqref{eq3.5}, we deduce that
\begin{eqnarray*}
 m_{1, \mu_j}&=& J_{\mu_j}(u_{\mu_j},v_{\mu_j})\\
 &=&\frac{1}{2}\int_{\mathbb{R}^3\times \mathbb{T} }\left(|  \nabla_xu_{\mu_j}|^{2}+\mu_j|u_{y,\mu_j}|^2+|  \nabla_xv_{\mu_j}|^{2}+\mu_j|v_{y,\mu_j}|^2\right)dxdy   \\
 &&-\frac{1}{4}\int_{\mathbb{R}^3\times\mathbb{T}}(|u_{\mu_j}|^{4}+2|u_{ \mu_j}|^2|v_{ \mu_j}|^2+|v_{\mu_j}|^{4})dxdy\\
    &\geq&\int_{ \mathbb{T} } \left(\frac{1}{2}\int_{\mathbb{R}^3}(|\nabla_xu_{\mu_j}|^2+|\nabla_xv_{\mu_j}|^2)dx-\frac{1}{4}\int_{\mathbb{R}^3}(|u_{ \mu_j}|^4+2|u_{ \mu_j}|^2|v_{ \mu_j}|^2+|v_{ \mu_j}|^4)dx\right)dy  \\
    &\geq& \widehat{m}_1   \int_{ \mathbb{T} }w_jdy=\widehat{m}_1\cdot2\pi\cdot(2\pi)^{-1}=2\pi\widehat{m}_{\frac{1}{\sqrt{2\pi}}}.
\end{eqnarray*}
Finally, we prove \eqref{eq3.3}. In fact, by using previous analysis, we get
\begin{equation*}
  2\pi\widehat{m}_{\frac{1}{\sqrt{2\pi}}}\geq m_{1, \mu_j}\geq\frac{1}{2}\int_{\mathbb{R}^3\times \mathbb{T} }\mu_j(|u_{y,\mu_j}|^2+|v_{y,\mu_j}|^2) dxdy+2\pi\widehat{m}_{\frac{1}{\sqrt{2\pi}}},
\end{equation*}
which complete the proof.
\end{proof}
\subsection{Dependency analysis}

\begin{lemma}\label{L3.3}
We have
\begin{eqnarray}\label{eq3.6}
   \int_{\mathbb{R}^3\times \mathbb{T} }(|\nabla_xu_{\mu_j}|^{2}+|\nabla_xv_{\mu_j}|^{2})dxdy =\frac{3}{4}\int_{\mathbb{R}^3\times\mathbb{T}}(|u_{\mu_j}|^{4}+2|u_{\mu_j}|^{2}|v_{\mu_j}|^{2}+|v_{\mu_j}|^{4})dxdy
\end{eqnarray}
for $\mu_j>0$. Then there exists $\lambda_{1,\mu}, \lambda_{2,\mu}\in\mathbb{R}$ such that
\begin{equation}\label{eq3.7}
   - \Delta_xu_{\mu}-\mu  u_{yy,\mu}+\lambda_{1,\mu}u_\mu=|u_\mu|^{2}u_\mu+u_\mu |v_\mu|^2 ,
\end{equation}
\begin{equation}\label{eq3.8}
   - \Delta_xv_{\mu}-\mu  v_{yy,\mu}+\lambda_{2,\mu}v_\mu=|v_\mu|^{2}v_\mu+v_\mu |u_\mu|^2.
\end{equation}
Moreover,
\begin{equation}\label{eq3.9}
  \lim\limits_{j\rightarrow\infty}\lambda_{1,\mu_j}=\overline{\lambda_1},\  \lim\limits_{j\rightarrow\infty}\lambda_{2,\mu_j}=\overline{\lambda_2}.
\end{equation}
\end{lemma}
\begin{proof}
Note that, $u_{\mu_j}$ is a constrained minimizer for $J_\mu(u)$  on the ball of size $1$ in $L^2(\mathbb{R}^3 \times\mathbb{T} )$, it follows that $\frac{d}{d\varepsilon}J_{\mu_j}(u_{\mu_j}^\varepsilon,v_{\mu_j}^\varepsilon)\big|_{\varepsilon=1} =0$, where $(u_{\mu_j}^\varepsilon,v_{\mu_j}^\varepsilon)=(\varepsilon^{\frac{3}{2}}u_{\mu_j}(\varepsilon x,y),\varepsilon^{\frac{3}{2}}u_{\mu_j}(\varepsilon x,y))$. Moreover, we have
\begin{eqnarray*}
  \frac{d}{d\varepsilon}J_{\mu_j}(u_{\mu_j}^\varepsilon,v_{\mu_j}^\varepsilon)\big|_{\varepsilon=1}&=&\frac{d}{d\varepsilon}\left[ \int_{\mathbb{R}^3\times \mathbb{T} }\left(\frac{\varepsilon^2}{2}|  \nabla_x u_{\mu_j}|^{2}+\frac{\mu_j }{2}| u_{y, \mu_j}|^2+\frac{\varepsilon^2}{2}|  \nabla_x v_{\mu_j}|^{2}+\frac{\mu_j }{2}| v_{y, \mu_j}|^2\right)dxdy \right. \\
  &&-\left.\frac{\varepsilon^{3}}{4}\int_{\mathbb{R}^3\times\mathbb{T}}(|u_{\mu_j}|^{4}+2|u_{\mu_j}|^2|v_{\mu_j}|^2+|v_{\mu_j}|^{4})dxdy \right] \\
  &=& \int_{\mathbb{R}^3\times \mathbb{T} }|\nabla_xu_{\mu_j}|^{2}dxdy -\frac{3}{4}\int_{\mathbb{R}^3\times\mathbb{T}}|u_{\mu_j}|^{4}dxdy+\int_{\mathbb{R}^3\times \mathbb{T} }|\nabla_xv_{\mu_j}|^{2}dxdy\\
  &&-\frac{3}{4}\int_{\mathbb{R}^3\times\mathbb{T}}|v_{\mu_j}|^{4}dxdy-\frac{3}{2}\int_{\mathbb{R}^3\times\mathbb{T}}|u_{\mu_j}|^{2}|v_{\mu_j}|^{2}dxdy,
\end{eqnarray*}
which implies \eqref{eq3.6} holds. \eqref{eq3.7} and \eqref{eq3.8} follow from Lagrange multiplier technique. Now, by using \eqref{eq3.7} and \eqref{eq3.8}, one has
\begin{equation*}
   \int_{\mathbb{R}^3 \times\mathbb{T} }(| \nabla_xu_{\mu_j}|^2+\mu_j|u_{y,\mu_j}|^2) d xdy +\lambda_{1,\mu_j}\|u_{\mu_j}\|_2^2= \int_{\mathbb{R}^3 \times\mathbb{T} }(|u_{\mu_j}|^{4}+|u_{\mu_j}|^{2}|v_{\mu_j}|^{2}) d xdy
\end{equation*}
and
\begin{equation*}
   \int_{\mathbb{R}^3 \times\mathbb{T} }(| \nabla_xv_{\mu_j}|^2+\mu_j|v_{y,\mu_j}|^2) d xdy +\lambda_{2,\mu_j}\|v_{\mu_j}\|_2^2=\int_{\mathbb{R}^3 \times\mathbb{T} }(|v_{\mu_j}|^{4}+|u_{\mu_j}|^{2}|v_{\mu_j}|^{2}) d xdy.
\end{equation*}
Hence, it follows from \eqref{eq3.6} and \eqref{eq3.3} that
\begin{eqnarray*}
&&\lambda_{1,\mu_j}+\lambda_{2,\mu_j}\\
&=& - \int_{\mathbb{R}^3\times\mathbb{T} }(| \nabla_xu_{\mu_j}|^2+| \nabla_xv_{\mu_j}|^2) d xdy+\int_{\mathbb{R}^3\times\mathbb{T} }(|u_{\mu_j}|^{4}+2|u_{\mu_j}|^{2}|v_{\mu_j}|^{2}+|v_{\mu_j}|^{4})  d xdy+o_j(1) \\
&=&-\int_{\mathbb{R}^3\times\mathbb{T} }(| \nabla_xu_{\mu_j}|^2+| \nabla_xv_{\mu_j}|^2) d xdy+\frac{4}{3}\int_{\mathbb{R}^d\times\mathbb{T} }(| \nabla_xu_{\mu_j}|^2+| \nabla_xv_{\mu_j}|^2) d xdy+o_j(1) \\
    &=&\frac{1}{3}\int_{\mathbb{R}^3\times\mathbb{T} }(| \nabla_xu_{\mu_j}|^2+| \nabla_xv_{\mu_j}|^2) d xdy+o_j(1),
\end{eqnarray*}
Note that,
 \begin{eqnarray*}
 &&J_{\mu_j}(u_{\mu_j},v_{\mu_j})\\
 &=&\frac{1}{2}\int_{\mathbb{R}^3\times\mathbb{T} }(| \nabla_xu_{\mu_j}|^2+| \nabla_xv_{\mu_j}|^2) d xdy-\frac{1}{4}\int_{\mathbb{R}^3\times\mathbb{T} }(|u_{\mu_j}|^{4}+2|u_{\mu_j}|^{2}|v_{\mu_j}|^{2}+|v_{\mu_j}|^{4})d xdy+o_j(1)\\
 &=&\frac{1}{6}\int_{\mathbb{R}^3\times\mathbb{T} }(| \nabla_xu_{\mu_j}|^2+| \nabla_xv_{\mu_j}|^2) d xdy+o_j(1),
 \end{eqnarray*}
 then
\begin{eqnarray*}
\int_{\mathbb{R}^3\times\mathbb{T} }(| \nabla_xu_{\mu_j}|^2+| \nabla_xv_{\mu_j}|^2) d xdy&=&6J_{\mu_j}(u_{\mu_j},v_{\mu_j})=6 m_{1, \mu_j}=6 \cdot2\pi\widehat{m}_{\frac{1}{\sqrt{2\pi}}}+o_j(1),
\end{eqnarray*}
which implies \eqref{eq3.9} holds.
\end{proof}
\begin{lemma}\label{L3.4}
There exists some $w_1,\ w_2 \in \widehat{S}_{\frac{1}{ \sqrt{2\pi}   }}$ such that, up to a subsequence and
$\mathbb{R}^3$-translations, $(u_\mu, v_\mu)$ converges strongly in $H^1(\mathbb{R}_x^3)\times H^1(\mathbb{R}_x^3)$ to $(w_1,w_2)$ as $\mu \rightarrow \infty ,  \widehat{I}(w_1, w_2) = \widehat{m}_{\frac{1}{ \sqrt{2\pi}  }}$, and
$w_1, w_2$ satisfies
\begin{equation} \label{eq3.10}
 \left\{\aligned
&-\Delta_{x}w_1+\overline{\lambda_1} w_1 =|w_1|^{2}w_1+w_1|w_2|^2, \\
&-\Delta_{x}w_2+\overline{\lambda_2} w_2 =|w_2|^{2}w_2+w_2|w_1|^2.
\endaligned
\right.
\end{equation}
\end{lemma}
\begin{proof}
We split the proof into many steps.

\textbf{Step 1} We prove that \eqref{eq3.10} holds. Indeed, it is well known that
\begin{equation*}
 \left\{\aligned
&-\Delta_{x}w_1+\widetilde{\lambda_1} w_1 =|w_1|^{2}w_1+w_1|w_2|^2\ \text{for a suitable}\ \widetilde{\lambda_1}, \\
&-\Delta_{x}w_2+\widetilde{\lambda_2} w_2 =|w_2|^{2}w_2+w_2|w_1|^2\ \text{for a suitable}\ \widetilde{\lambda_2}.
\endaligned
\right.
\end{equation*}
We must prove that $(\overline{\lambda_1},\overline{\lambda_2})=(\widetilde{\lambda_1},\widetilde{\lambda_2})$. Note that $(\widetilde{\lambda_1}+\widetilde{\lambda_2})\int_{\mathbb{R}^3}|u|^2dx=2\widehat{m}_{\|u\|_2}$, then
\begin{equation*}
  (\widetilde{\lambda_1}+\widetilde{\lambda_2})\cdot\frac{1}{2\pi}=2\widehat{m}_{\|u\|_2}=2\widehat{m}_{\frac{1}{\sqrt{2\pi}}}.
\end{equation*}
Moreover, by using Lemma \ref{L3.3},
\begin{eqnarray*}
  \lambda_{1,\mu_j}+\lambda_{2,\mu_j}&=&\frac{1}{3}\int_{\mathbb{R}^3\times\mathbb{T} }(| \nabla_xu_{\mu_j}|^2+| \nabla_xv_{\mu_j}|^2) d xdy+o_j(1)=2\cdot2\pi\widehat{m}_{\frac{1}{\sqrt{2\pi}}}+o_j(1),
\end{eqnarray*}
passing to the limit in $j$, then $(\overline{\lambda_1},\overline{\lambda_2})=(\widetilde{\lambda_1},\widetilde{\lambda_2})$.

\textbf{Step 2}
According to \eqref{eq3.3}, $\int_{\mathbb{R}^3\times\mathbb{T} }(| \nabla_xu_{\mu_j}|^2+| \nabla_xv_{\mu_j}|^2) d xdy=6\cdot 2\pi\widehat{m}_{\frac{1}{\sqrt{2\pi}}}+o_j(1)$ and $\left\|u_{\mu_j}\right\|_{L_{x, y}^2}=\left\|v_{\mu_j}\right\|_{L_{x, y}^2}=1$, we know that $(u_{\mu_j},v_{\mu_j})$ is bounded in $H^1(\mathbb{R}^3\times \mathbb{T})\times H^1(\mathbb{R}^3\times \mathbb{T})$. Moreover, by the localized Gagliardo-Nirenberg inequality,  there exists  $\tau_j \in \mathbb{R}_x^3 $ such that
$$
(u_{\mu_j}\left(x+\tau_j, y\right),v_{\mu_j}\left(x+\tau_j, y\right)) \rightharpoonup (w_1,w_2) \neq 0 \ \text { in }  H^1(\mathbb{R}^3\times \mathbb{T})\times H^1(\mathbb{R}^3\times \mathbb{T}).
$$
It follows from  \eqref{eq3.3} that $(w_1,w_2)$ is $y$-independent. Taking the limit in \eqref{eq3.7} and \eqref{eq3.8} in the distribution sense, we obtain that
\begin{equation} \label{eq3.11}
 \left\{\aligned
&-\Delta_{x}w_1+\overline{\lambda_1} w_1 =|w_1|^{2}w_1+w_1|w_2|^2, \\
&-\Delta_{x}w_2+\overline{\lambda_2} w_2 =|w_2|^{2}w_2+w_2|w_1|^2.
\endaligned
\right.
\end{equation}
Now, we prove that $\|w_1\|_{L_x^2}=\|w_2\|_{L_x^2}=\frac{1}{ \sqrt{2\pi} }$. Indeed, we can assume $\|w_1\|_{L_x^2}=\|w_2\|_{L_x^2}=\sigma<\frac{1}{ \sqrt{2\pi}}(\|w_1\|_{L_x^2}=\sigma_1<\|w_2\|_{L_x^2}=\sigma_2<\frac{1}{ \sqrt{2\pi}})$. Using $w_1,\ w_2$ are the solution of \eqref{eq3.7} and \eqref{eq3.8}, similar to the proof in \cite{HLTL2026}, we obtain contradiction. Hence,  $\sigma=\frac{1}{ \sqrt{2\pi}}(\sigma_1=\sigma_2=\frac{1}{ \sqrt{2\pi}})$.

\textbf{Step 3} We prove that $u_{\mu_j}\left(x+\tau_j, y\right)$ converges strongly to $w_1$ in $H^1(\mathbb{R}^3\times \mathbb{T})$. In fact, by Lemma \ref{L3.3} and $w_{1,y}=0$, one has
$$
\lim _{j \rightarrow \infty}\left\| u_{y,\mu_j}\left(x+\tau_j, y\right)\right\|_{L_{x, y}^2}=0=\left\|w_{1,y}\right\|_{L_{x, y}^2}.
$$
Hence, it is sufficient to obtain that
$$
\lim _{j \rightarrow \infty}\left\| u_{x,\mu_j}\left(x+\tau_j, y\right)\right\|_{L_{x, y}^2}=\sqrt{2\pi}\left\|w_{1,x}\right\|_{L_x^2}=\left\|w_{1,x}\right\|_{L_{x, y}^2}.
$$
Similar to the proof of Theorem \ref{t1.1}, the last fact follows by combining $\|w_1\|_{L_x^2}=\frac{1}{\sqrt{2\pi}}$ and \eqref{eq3.10}. Similarly, $v_{\mu_j}\left(x+\tau_j, y\right)$ converges strongly to $w_2$ in $H^1(\mathbb{R}^3\times \mathbb{T})$.
\end{proof}
\begin{lemma}\label{L3.5}
There exists $j_0$ such that $(u_{y, \mu_j},v_{y, \mu_j})=0$ for all $j>j_{0}$.
\end{lemma}
\begin{proof}
Under the stimulation of \cite{STNT2014}. We introduce $(w_{1,\mu_j},w_{2,\mu_j})=(\sqrt{- \partial_{yy} }u_{\mu_j},\sqrt{- \partial_{yy} }v_{\mu_j})$, note that for all $p \in(1, \infty)$ there exist $c(p), C(p)>0$ such that
\begin{equation}\label{eq3.12}
  c(p)\left\|\sqrt{- \partial_{yy} }u_{\mu_j}\right\|_{L_y^p} \leq\left\| u_{y,\mu_j}\right\|_{L_y^p} \leq C(p)\left\|\sqrt{- \partial_{yy} }u_{\mu_j}\right\|_{L_y^p} .
\end{equation}
Applying $\sqrt{- \partial_{yy} }$ to \eqref{eq3.7} and \eqref{eq3.8}, one has
$$
 -\mu_j  w_{yy,1,\mu_j}-\Delta_xw_{1,\mu_j}+\lambda_{1,\mu_j} w_{1,\mu_j}=\sqrt{- \partial_{yy} }\left(u_{\mu_j}^{3}+u_{\mu_j}v_{\mu_j}^2 \right)
$$
and
$$
 -\mu_j  w_{yy,2,\mu_j}-\Delta_xw_{2,\mu_j}+\lambda_{2,\mu_j} w_{2,\mu_j}=\sqrt{- \partial_{yy} }\left(v_{\mu_j}^{3}+v_{\mu_j}u_{\mu_j}^2 \right).
$$
After multiplication by $w_{1,\mu_j}$ and $w_{2,\mu_j}$, we know that
$$
\int_{\mathbb{R}^3 \times\mathbb{T} }  \left[\mu_j\left|  w_{y,1, \mu_j}\right|^2+\left|  \nabla_xw_{1,\mu_j}\right|^2+\lambda_{1,\mu_j}\left|w_{1,\mu_j}\right|^2-\sqrt{- \partial_{yy} }\left(u_{\mu_j}^{3}+v_{\mu_j}u_{\mu_j}^2\right)w_{1,\mu_j}\right] d x dy=0
$$
and
$$
\int_{\mathbb{R}^3 \times\mathbb{T} }  \left[\mu_j\left|  w_{y,2, \mu_j}\right|^2+\left|  \nabla_xw_{2,\mu_j}\right|^2+\lambda_{2,\mu_j}\left|w_{2,\mu_j}\right|^2-\sqrt{- \partial_{yy} }\left(v_{\mu_j}^{3}+u_{\mu_j}v_{\mu_j}^2\right)w_{2,\mu_j}\right] d x dy=0.
$$
By direct calculation,
\begin{eqnarray*}
&& \int_{\mathbb{R}^3 \times\mathbb{T} }\left[\left(\mu_j-1\right)\left| w_{y,1,\mu_j}\right|^2- \sqrt{- \partial_{yy} }\left(3u_{\mu_j} |w_{1}|^{2}+v_{\mu_j} |w_{1}|^{2}+2w_2w_1u_{\mu_j} \right) w_{1,\mu_j}  \right]d x d y \\
& +& \int_{\mathbb{R}^3 \times\mathbb{T} }\left[\left(\mu_j-1\right)\left| w_{y,2,\mu_j}\right|^2- \sqrt{- \partial_{yy} }\left(3v_{\mu_j} |w_{2}|^{2}+u_{\mu_j} |w_{2}|^{2}+2w_1w_2v_{\mu_j} \right) w_{2,\mu_j}  \right]d x d y  \\
& +&\int_{\mathbb{R}^3 \times\mathbb{T} }\left[\left| w_{y,1, \mu_j}\right|^2+\left| w_{y,2, \mu_j}\right|^2+\left|\nabla_xw_{1,\mu_j}\right|^2+\left|\nabla_xw_{2,\mu_j}\right|^2+\overline{\lambda}_1\left|w_{1,\mu_j}\right|^2+\overline{\lambda}_2\left|w_{2,\mu_j}\right|^2\right] d x d y \\
&+&\int_{\mathbb{R}^3 \times\mathbb{T} }\sqrt{- \partial_{yy} }\left(3u_{\mu_j} |w_{1}|^{2}+v_{\mu_j} |w_{1}|^{2}+2w_2w_1u_{\mu_j}-u_{\mu_j}^{3}-v_{\mu_j}u_{\mu_j}^2\right) w_{1,\mu_j}dxdy\\
&+&\int_{\mathbb{R}^3 \times\mathbb{T} }\sqrt{- \partial_{yy} }\left(3v_{\mu_j} |w_{2}|^{2}+u_{\mu_j} |w_{2}|^{2}+2w_1w_2v_{\mu_j}-v_{\mu_j}^{3}+u_{\mu_j}v_{\mu_j}^2\right) w_{2,\mu_j}dxdy\\
& +&\int_{\mathbb{R}^3 \times\mathbb{T} }(\lambda_{1,\mu_j}-\overline{\lambda}_1)\left|w_{1,\mu_j}\right|^2 d x d y+\int_{\mathbb{R}^3 \times\mathbb{T} }(\lambda_{2,\mu_j}-\overline{\lambda}_2)\left|w_{2,\mu_j}\right|^2 d x d y \\
 &\equiv& I_{\mu_j}+II_{\mu_j}+III_{\mu_j}+IV_{\mu_j}+V_{\mu_j}+VI_{\mu_j}=0.
\end{eqnarray*}
Next we fix an orthonormal basis of eigenfunctions for $- \partial_{yy}$, that is, $- \partial_{yy} \varphi_k=\rho_k \varphi_k$ and $\varphi_0=C$, where $C$ is a constant.

In order to estimate $I_{\mu_j}$ and $II_{\mu_j}$, we write
\begin{equation}\label{eq3.13}
  w_{1,\mu_j}(x, y)= \sum_{k \in \mathbb{N}  \backslash\{0\}} a_{1,\mu_j,k}(x) \varphi_{1,k}(y),
\end{equation}
and
\begin{equation}\label{eq3.14}
  w_{2,\mu_j}(x, y)= \sum_{k \in \mathbb{N}  \backslash\{0\}} a_{2,\mu_j,k}(x) \varphi_{2,k}(y),
\end{equation}
where the eigenfunction $\varphi_0$ does not enter in the development. Using standard elliptic regularity theory, we are able to show that
$(w_1,w_2) \in L^\infty_x$. By using \eqref{eq3.13} and \eqref{eq3.14}, one has
\begin{eqnarray*}
&&I_{\mu_j} +II_{\mu_j}\\
&=&\int_{\mathbb{R}^3 \times\mathbb{T} }\left[\left(\mu_j-1\right)\left| w_{y,1,\mu_j}\right|^2- \sqrt{- \partial_{yy} }\left(3u_{\mu_j} |w_{1}|^{2}+v_{\mu_j} |w_{1}|^{2}+2w_2w_1u_{\mu_j} \right) w_{1,\mu_j}  \right]d x d y\\
&&+\int_{\mathbb{R}^3 \times\mathbb{T} }\left[\left(\mu_j-1\right)\left| w_{y,2,\mu_j}\right|^2- \sqrt{- \partial_{yy} }\left(3v_{\mu_j} |w_{2}|^{2}+u_{\mu_j} |w_{2}|^{2}+2w_1w_2v_{\mu_j} \right) w_{2,\mu_j}  \right]d x d y\\
&\geq&(\mu_j-1)\int_{\mathbb{R}^3 \times\mathbb{T} } \left|  w_{y,1, \mu_j}\right|^2dxdy-C(\|w_1\|_{L^\infty},\|w_2\|_{L^\infty})\int_{\mathbb{R}^3 \times\mathbb{T} } (w_{1,\mu_j}w_{2,\mu_j}+w_{1,\mu_j}^2) d x d y \\
&&+(\mu_j-1)\int_{\mathbb{R}^3 \times\mathbb{T} } \left|  w_{y,2, \mu_j}\right|^2dxdy-C(\|w_1\|_{L^\infty},\|w_2\|_{L^\infty})\int_{\mathbb{R}^3 \times\mathbb{T} } (w_{1,\mu_j}w_{2,\mu_j}+w_{2,\mu_j}^2) d x d y \\
&\geq&(\mu_j-1)\int_{\mathbb{R}^3 \times\mathbb{T} } \left|  w_{y,1, \mu_j}\right|^2dxdy-C(\|w_1\|_{L^\infty},\|w_2\|_{L^\infty})\int_{\mathbb{R}^3 \times\mathbb{T} } w_{1,\mu_j}^2  d x d y \\
&&+(\mu_j-1)\int_{\mathbb{R}^3 \times\mathbb{T} } \left|  w_{y,2, \mu_j}\right|^2dxdy-C(\|w_1\|_{L^\infty},\|w_2\|_{L^\infty})\int_{\mathbb{R}^3 \times\mathbb{T} } w_{2,\mu_j}^2  d x d y \\
&\geq& C\sum_{k \neq 0}\left(\mu_j-1\right)\left|\rho_k\right|^2 \int_{\mathbb{R}^3 }\left|  a_{1,\mu_j,k}(x)\right|^2 d x- C(\|w_1\|_{L^\infty},\|w_2\|_{L^\infty})\sum_{k \neq 0} \int_{\mathbb{R}^3 } \left| a_{1,\mu_j,k}(x)\right|^2 d x \\
&&+ C\sum_{k \neq 0}\left(\mu_j-1\right)\left|\rho_k\right|^2 \int_{\mathbb{R}^3 }\left|  a_{2,\mu_j,k}(x)\right|^2 d x- C(\|w_1\|_{L^\infty},\|w_2\|_{L^\infty})\sum_{k \neq 0} \int_{\mathbb{R}^3 } \left| a_{2,\mu_j,k}(x)\right|^2 d x \\
&\geq&0
\end{eqnarray*}
for $\mu_j\gg1$. It follows from \eqref{eq3.9} that
$$
 VI_{\mu_j}=\int_{\mathbb{R}^3 \times\mathbb{T} }(\lambda_{1,\mu_j}-\overline{\lambda}_1)\left|w_{1,\mu_j}\right|^2 d x d y+\int_{\mathbb{R}^3 \times\mathbb{T} }(\lambda_{2,\mu_j}-\overline{\lambda}_2)\left|w_{2,\mu_j}\right|^2 d x d y\rightarrow0 \ \text{as}\ \mu_j\rightarrow+\infty.
$$
Hence, one has $I_{\mu_j}+II_{\mu_j}+VI_{\mu_j} \geq 0$ for $\mu_j$ large enough. In order to estimate $I II_{\mu_j}$ and $IV_{\mu_j}$, notice that, by the Cauchy-Schwartz inequality and \eqref{eq3.12} we get
\begin{eqnarray*}
&&\int_{\mathbb{R}^3 \times\mathbb{T} }\sqrt{- \partial_{yy} }\left(3u_{\mu_j} |w_{1}|^{2}+v_{\mu_j} |w_{1}|^{2}+2w_2w_1u_{\mu_j}-u_{\mu_j}^{3}-v_{\mu_j}u_{\mu_j}^2\right) w_{1,\mu_j}dxdy \\
&\leq&\int_{\mathbb{R}^3 \times\mathbb{T} } 3[(|w_{1}|^{2} -u_{\mu_j}^{2}) w_{1,\mu_j}^2+w_{1,\mu_j}w_{2,\mu_j}( |w_{1}|^{2}-u_{\mu_j}^2)+2w_{1,\mu_j}^2(w_2w_1-v_{\mu_j}u_{\mu_j})]dxdy \\
&\leq&C\int_{\mathbb{R}^3 \times\mathbb{T} } |w_{1}  -u_{\mu_j}|(w_{1}+u_{\mu_j})  w_{1,\mu_j}^2 dxdy+C\int_{\mathbb{R}^3 \times\mathbb{T} } |w_{1}  -u_{\mu_j}|(w_{1}+u_{\mu_j})  w_{1,\mu_j}w_{2,\mu_j}  dxdy \\
&&+C\int_{\mathbb{R}^3 \times\mathbb{T} } w_{1,\mu_j}^2(w_2-v_{\mu_j})w_1  dxdy+C\int_{\mathbb{R}^3 \times\mathbb{T} } w_{1,\mu_j}^2(w_1-u_{\mu_j})v_{\mu_j}  dxdy\\
&\leq& C\|w_1-u_{\mu_j}\|_{L_{x, y}^{4}} (\|w_1\|_{L_{x, y}^{4}} +\|u_{ \mu_j}\|_{L_{x, y}^{4}} ) \|w_{1,\mu_j}\|_{L_{x, y}^{4}}^2+C\|w_1-u_{\mu_j}\|_{L_{x, y}^{4}} \|v_{\mu_j}\|_{L_{x, y}^{4}}  \|w_{1,\mu_j}\|_{L_{x, y}^{4}}^2\\
&&+C\|w_1-u_{\mu_j}\|_{L_{x, y}^{4}} (\|w_1\|_{L_{x, y}^{4}} +\|u_{\mu_j}\|_{L_{x, y}^{4}} ) \|w_{1,\mu_j}\|_{L_{x, y}^{4}}\|w_{2,\mu_j}\|_{L_{x, y}^{4}}\\
&&+C\|w_2-v_{\mu_j}\|_{L_{x, y}^{4}} \|w_{1}\|_{L_{x, y}^{4}}  \|w_{1,\mu_j}\|_{L_{x, y}^{4}}^2 .
\end{eqnarray*}
By Lemma \ref{L3.4}, we know that $(u_{\mu_j},v_{\mu_j})\rightarrow (w_1,w_2)$ in $H^1(\mathbb{R}^3 \times\mathbb{T})\times H^1(\mathbb{R}^3 \times\mathbb{T})$.
Therefore, it follows from Sobolev embedding $H^1(\mathbb{R}^3 \times\mathbb{T})\hookrightarrow L_{x, y}^{p}$  that
\begin{equation*}
  III_{\mu_j}+IV_{\mu_j} \rightarrow0 \ \text{as}\ \mu_j\rightarrow+\infty.
\end{equation*}
By combining this information and the structure of $I I_{\mu_j}$, one has
$$
0\geq I I_{\mu_j} \geq 0 \quad \text { for } j>j_0,
$$
so we deduce $(w_{1, \mu_j},w_{2, \mu_j})=0$ for $\mu_j$ large enough.
\end{proof}
\begin{lemma}\label{L3.6}
There exists some $\mu^* \in (0,\infty )$ such that

(1) for all $\mu \in (0,\mu^*)$ we have $m_{1,\mu} <  2\pi   \widehat{m}_{\frac{1}{\sqrt{2\pi}} }$. Moreover, for $\mu \in (0, \mu^*)$
any minimizer $(u_\mu,v_\mu)$ of $m_{1,\mu}$ satisfies $(u_{y,\mu},v_{y,\mu})\neq 0$;

(2) for all $\mu \in (\mu^*,\infty)$ we have $m_{1,\mu}=  2\pi   \widehat{m}_{\frac{1}{\sqrt{2\pi}} } $. Moreover, for $\mu \in (\mu^* ,\infty )$
any minimizer $(u_\mu,v_\mu)$ of $m_{1,\mu}$ satisfies $(u_{y,\mu},v_{y,\mu})= 0$.
\end{lemma}
\begin{proof}
Define
\begin{equation*}
  \mu^*:= \inf\{ \mu \in (0,\infty ) : m_{1,\mu} = 2\pi   \widehat{m}_{\frac{1}{\sqrt{2\pi}} },\  \forall  \mu\geq\mu^*\} .
\end{equation*}
The existence of $\mu^*$ follows by Lemmas \ref{L3.1} and \ref{L3.2}. The fact that a minimizer of
$m_{1,\mu}$ has nontrivial $y$-dependence and $m_{1,\mu} < 2\pi   \widehat{m}_{\frac{1}{\sqrt{2\pi}} }$ for $\mu < \mu^*$ follows already
from the definition of $\mu^*$ and the fact that $\mu\mapsto m_{1,\mu}$ is monotone increasing.

Now, we claim that any minimizer of $m_{1,\mu}$ for $\mu > \mu^*$  must be $y$-independent. In fact, assume that an optimizer $(u_{\mu},v_{\mu})$ of $m_{1,\mu}$ satisfies $\|u_\mu\|_2^2 \neq 0$ and $\|v_\mu\|_2^2 \neq 0$.
Then there exists some $\mu^{**} \in (\mu^*, \mu)$. Consequently,
\begin{equation*}
  2\pi   \widehat{m}_{\frac{1}{\sqrt{2\pi}}} = m_{1,\mu^{**}} \leq J_{\mu^{**}} (u_{\mu } ) = J_\mu (u_\mu) + \frac{\mu^{**}-\mu}{2} (\| u_{y,\mu} \|_2^2+\| v_{y,\mu} \|_2^2)
< J_\mu  (u_\mu ) = m_{1,\mu} = 2\pi   \widehat{m}_{\frac{1}{\sqrt{2\pi}}},
\end{equation*}
which is a contradiction. We also get contradiction by similar argument when $\|u_\mu\|_2^2 \neq 0(\|v_\mu\|_2^2=0)$ and $\|v_\mu\|_2^2 \neq 0(\|u_\mu\|_2^2 = 0)$.
\end{proof}
\noindent\textbf{Proof of Theorem \ref{t1.2}}  By computation, we have that the map
 \begin{equation*}
    S_1\ni u\rightarrow \tau^{-2}u(\tau^{-2}x,y) \in S_\tau
 \end{equation*}
where
\begin{equation*}
  S_\tau = \{v \in H^1(\mathbb{R}^3 \times\mathbb{T} ) : \|v\|_{L^2_{x,y}} = \tau\}
\end{equation*}
is a bijection. Moreover,
 \begin{eqnarray*}
m_\tau&=&I(\tau^{-2}u(\tau^{-2}x,y),\tau^{-2}v(\tau^{-2}x,y))\\
&=&\tau^{-2}\int_{\mathbb{R}^3\times \mathbb{T} }\left[\frac{1}{2}\left(|  \nabla_xu|^{2}+\tau^{4}|u_{y}|^{2}+|  \nabla_xv|^{2}+\tau^{4}|v_{y}|^{2}\right) -\frac{1}{4} (|u|^{4}+2|u|^2|v|^2+|v|^{4})\right] dxdy \\
&=&\tau^{-2}m_{1,\tau^{4}}.
 \end{eqnarray*}
It follows from the same rescaling arguments that $\widehat{m}_{\frac{\tau}{\sqrt{2\pi}}}=\tau^{-2}\widehat{m}_{\frac{1}{\sqrt{2\pi}}}$ for $\tau> 0$. Thus, by Lemma \ref{L3.6}, there exists $\tau_0\in(0,+\infty)$ such that

(1) for all $\tau \in (0,\tau_0)$ we have
\begin{equation*}
  m_\tau=\tau^{-2}m_{1,\tau^{4}}<\tau^{-2}2\pi\widehat{m}_{\frac{1}{\sqrt{2\pi}}}=2\pi\widehat{m}_{\frac{\tau}{\sqrt{2\pi}}}.
\end{equation*}

(2) for all $\tau \in (\tau_0,+\infty)$ we have
\begin{equation*}
  m_\tau=\tau^{-2}m_{1,\tau^{4}}=\tau^{-2}2\pi\widehat{m}_{\frac{1}{\sqrt{2\pi}}}=2\pi\widehat{m}_{\frac{\tau}{\sqrt{2\pi}}}.
\end{equation*}
This completes the proof.

\textbf{Data availability} This article has no additional data.

\textbf{Conflict of interest} On behalf of all authors, the corresponding author states that there is no Conflict of
interest.

\end{document}